\documentclass{amsart}
\usepackage{graphicx}
\usepackage{trajan}
\usepackage{multirow}
\usepackage{latexsym,array,delarray,amsthm,amssymb,epsfig}
\usepackage{amsmath}
\usepackage{yhmath}
\usepackage{mathdots}
\usepackage{MnSymbol}
\usepackage{tikz}
\usepackage{adjustbox}
\usepackage{rotating}
\usepackage{longtable}
\usepackage{enumitem}
\usepackage{caption}
\usetikzlibrary{shapes.misc,shadows}
\theoremstyle{plain}
\newtheorem{thm}{Theorem}[section]
\newtheorem{lemma}[thm]{Lemma}

\newtheorem*{thm*}{Theorem}
\newtheorem*{lemma*}{Lemma}
\newtheorem*{prop*}{Proposition}
\newtheorem*{cor*}{Corollary}
\newtheorem*{conj*}{Conjecture}

\theoremstyle{definition}

\theoremstyle{remark}
\newtheorem*{rmk}{Remark}

\newcommand{\cc}{\mathbb{C}}

\begin{document}
\date{}

\title{Classification of some Solvable Leibniz Algebras}
\author{Ismail Demir, Kailash C. Misra and Ernie Stitzinger}
\address{Department of Mathematics, North Carolina State University, Raleigh, NC 27695-8205}
\email{idemir@ncsu.edu, misra@ncsu.edu, stitz@ncsu.edu}
\thanks{KCM is partially supported by Simons Foundation grant \#  307555}

\begin{abstract}
Leibniz algebras are certain generalization of Lie algebras. In this paper we give classification of non-Lie solvable (left) Leibniz algebras of dimension $\leq 8$ with one dimensional derived subalgebra. We use the canonical forms for the congruence classes of matrices of bilinear forms to obtain our result. Our approach can easily be extended to classify these algebras of higher dimensions. We also revisit the classification of three dimensional non-Lie solvable (left) Leibniz algebras.

\end{abstract}

\maketitle
\bigskip
\section{Introduction}
Leibniz algebras introduced by Loday \cite{lodayfr} are natural generalization of Lie algebras. Earlier, such algebraic structures had been considered by A. Bloh who called them $D$-algebras \cite{bloh}. In this paper we assume the field to be $\mathbb{C}$, the field of complex numbers. A left (resp. right) Leibniz algebra $A$ is a $\mathbb{C}$-vector space equipped with a bilinear product $[ \, , \,] : A \times A \longrightarrow A$ such that the left (resp. right) multiplication operator is a derivation. A left Leibniz algebra is not necessarily a right Leibniz algebra. In this paper, following Barnes \cite{barnes1} we consider (left) Leibniz algebras. Of course, a Lie algebra is a Leibniz algebra, but not conversely.  For a Leibniz algebra $A$, we define the ideals $A^1 = A = A^{(1)}$,  $A^i = [A, A^{i-1}]$ and $A^{(i)} = [A^{(i-1)}, A^{(i-1)}]$ for $i \in \mathbb{Z}_{\geq 2}$. The Leibniz algebra $A$ is said be  abelian if $A^2 = 0$. It is nilpotent (resp. solvable) if $A^m = 0$ (resp. $A^{(m)} = 0$) for some positive integer $m$. An important abelian ideal of $A$ is $Leib(A)={\rm span}\{[a, a] \mid a\in A\}$. $A$ is a Lie algebra if and only if $Leib(A)={\rm span}\{[a, a] \mid a\in A\} = \{0\}$. The classification of non-Lie Leibniz algebras is still an open problem. So far the classification of non-Lie Leibniz algebras over $\cc$ of dimension less than or equal to four is known (see \cite{lodayfr}, \cite{3dim}, \cite{cil}, \cite{3comp}, \cite{our}, \cite{fourdimnil}, \cite{fourdim}). Most of these classification results are known for right Leibniz algebras.

In \cite{our} we introduced a technique of using canonical forms for the congruence classes of matrices of bilinear forms in the  classification of three dimensional non-Lie nilpotent Leibniz algebras. In this paper our main goal is to extend this technique to classify non-Lie solvable Leibniz algebras with one dimensional derived subalgebra. Our approach is simple and straightforward. We say that a Leibniz algebra $A$ is split if it can be written as a direct sum of two nontrivial ideals. Otherwise $A$ is said to be non-split. It suffices to determine the isomorphism classes of non-split Leibniz algebras. In particular, we show that the only isomorphism class of non-Lie non-split non-nilpotent solvable Leibniz algebra with a one dimensional derived subalgebra is the two dimensional solvable cyclic Leibniz algebra. After completion of this paper it was brought to our attention that a general classification result for (right) Leibniz algebras with one dimensional derived subalgebra is given in (\cite{AM}, Theorem 3.4) . However, the explicit isomorphism classes are hard to determine using this general result. 

In the last section we revisit the classification of three dimensional non-Lie solvable Leibniz algebras. Our approach is slightly different from that of \cite{3dim}, \cite{cil} and \cite{3comp}.

\maketitle
\bigskip
\section{Classification of solvable non-Lie Leibniz algebras \\ with one dimensional derived subalgebra}

First we consider finite dimensional non-Lie nilpotent Leibniz algebras.
Let $A$ be a $n$-dimensional non-Lie nilpotent Leibniz algebra with $\dim(A^2)=1$. Then $A^2=Leib(A)$ and $A^2={\rm span}\{x_n\}$ for some $0\neq x_n\in A$. Let $V$ be the complementary subspace to $A^2$ in $A$. Then for any $u, v\in V$, we have $[u, v]=cx_n$ for some $c\in \cc$. Define the bilinear form $f( \ , \ ): V\times V\rightarrow \cc$ by $f(u, v)=c$  for all $u, v\in V$. The canonical forms for the congruence classes of matrices associated with the bilinear form $f(\ ,\ )$ given in \cite{congr} is as follows. We denote

\begin{scriptsize}
\begin{center} $[A\backslash B]:= \left( \begin{array}{cc}
0 & B\\
A & 0\\
\end{array} \right)$
\end{center}

\end{scriptsize}

\begin{thm}\cite{congr} \label{cong} The matrix of the bilinear form $f( \ , \ ): V\times V\rightarrow \cc$ is congruent to a direct sum, uniquely determined up to permutation of summands, of canonical matrices of the following types:

\begin{scriptsize}

\begin{enumerate}
\item $A_{2k+1}=\left[ \begin{array}{cc} 
\left[ \begin{array}{cccc}
0 & 1 & & \\
  & \ddots & \ddots & \\
 & & 0 & 1\\
\end{array} \right] \backslash 
\left[ \begin{array}{cccc}
1 &  &  \\
0 & \ddots &  \\
  & \ddots & 1\\
 & & 0 \\
\end{array} \right]
\end{array} \right]_{(2k+1) \times (2k+1)}$

\item  $B_{2k}(c)=\left[ \begin{array}{cc} 
\left[ \begin{array}{cccc}
0 &  & & c \\
  &  & c & 1 \\
  & \udots & \udots & \\
c & 1 & & 0\\
\end{array} \right] \backslash 
\left[ \begin{array}{cccc}
 &  &  & 1 \\
 &  & 1 & c \\
  & \udots & \udots & \\
 1 & c & & 0 \\
\end{array} \right]
\end{array} \right]_{2k \times 2k},   \quad   c\neq \pm 1$.

\item $C_{2k+1}=
\left[ \begin{array}{ccccccc}
0 &  & & & & & 1 \\
  &  &  & & & 1 & 1 \\
 & & &  & \udots & \udots & \\
 &  & & 1 & 1 & & \\
 & & 1 & -1 & & & \\
 & \udots & \udots & & & & \\
1 & -1 &  & & &  & 0 \\
\end{array} \right]_{(2k+1) \times (2k+1)}$

\item  $D_{2k}=\left[ \begin{array}{cc} 
\left[ \begin{array}{cccc}
0 &  & & 1 \\
  &  & 1 & -1 \\
  & \udots & \udots & \\
1 & -1 & & 0\\
\end{array} \right] \backslash 
\left[ \begin{array}{cccc}
 &  &  & 1 \\
 &  & 1 & 1 \\
  & \udots & \udots & \\
 1 & 1 & & 0 \\
\end{array} \right]
\end{array} \right]_{2k \times 2k}$       \quad ($k$ even)

\item $E_{2k}=
\left[ \begin{array}{ccccccc}
0 &  & & & & & 1 \\
  &  &  & & & 1 & 1 \\
 & & &  & \udots & \udots & \\
 &  & & 1 & 1 & & \\
 & & -1 & 1 & & & \\
 & \udots & \udots & & & & \\
-1 & 1 &  & & &  & 0 \\
\end{array} \right]_{2k \times 2k}$

\item  $F_{2k}=\left[ \begin{array}{cc} 
\left[ \begin{array}{cccc}
0 &  & & -1 \\
  &  & -1 & 1 \\
  & \udots & \udots & \\
-1 & 1 & & 0\\
\end{array} \right] \backslash 
\left[ \begin{array}{cccc}
 &  &  & 1 \\
 &  & 1 & 1 \\
  & \udots & \udots & \\
 1 & 1 & & 0 \\
\end{array} \right]
\end{array} \right]_{2k \times 2k}$       \quad ($k$ odd)
\end{enumerate}
\end{scriptsize}
\end{thm}

Using Theorem \ref{cong}, we choose a basis $\{x_1, x_2, \cdots , x_{n-1}\}$ for $V$ so that the matrix of the bilinear form $f( \ , \ ): V\times V\rightarrow \cc$ is the $(n-1) \times (n-1)$ matrix 
$N$ given in Theorem \ref{cong}. Then $A$ has basis $\{x_1, x_2, \cdots , x_{n-1}, x_n\}$ and the multiplication among the basis vectors is completely determined by the matrix $N$ since $Leib(A) ={\rm span}\{x_n\}$.

\begin{lemma}
The matrix of $f( \ , \ )$ is of the form $N=K\oplus0$ if and only if the associated nilpotent Leibniz algebra $A$ with ${\rm dim}(A^2) = 1$ is split. 
\end{lemma}

\begin{proof} Suppose the matrix of $f( \ , \ )$ is of the form $N=K\oplus0$ where $K$ is a $k \times k$ matrix. Set $I_1 = {\rm span}\{x_1, \cdots , x_k , x_n\}$ and $I_2 = {\rm span}\{x_{k+1}, \cdots ,  x_{n-1}\}$. Then $I_1$ and $I_2$ are ideals of $A$ and $A = I_1 \oplus I_2$. So $A$ is split.

Conversely, suppose $A$ is split. Then $A = I_1 \oplus I_2$ where $I_1, I_2$ are ideals of $A$. Without loss of generality we can assume that $A^2 = Leib(A) = {\rm span}\{x_n\}$ is contained in $I_1$. Then $[I_2 , I_2] \subseteq A^2 \cap I_2 = \{0\}$ which implies that $I_2$ is abelian. Hence the matrix $N = K \oplus 0$ for some $k \times k$ matrix $K$, $ k < n-1$.
\end{proof}

In \cite{our} we used the technique of congruence matrices to determine the isomorphism classes of non-Lie nilpotent Leibniz algebras of dimension three. In this paper we extend this technique to determine the isomorphism classes of non-Lie nilpotent Leibniz algebras of dimension $\leq 8$ with one dimensional derived subalgebra. It suffices to determine the isomorphism classes of non-split Leibniz algebras.

\begin{thm}
Let $A$ be a non-split non-Lie nilpotent Leibniz algebra of dimension $4$ with $\dim(A^2)=1$. Then $A$ is isomorphic to a Leibniz algebra spanned by $\{x_1, x_2, x_3, x_4\}$ with the nonzero products given by one of the following:
\begin{description}
\item[1] $[x_1, x_3]= x_4, [x_3, x_2]=x_4$.
\item[2] $[x_1, x_3]= x_4, [x_2, x_2]=x_4, [x_2, x_3]=x_4, [x_3, x_1]=x_4, [x_3, x_2]=-x_4$.
\item[3] $[x_1, x_2]=x_4, [x_2, x_1]=-x_4, [x_3, x_3]=x_4$.
\item[4] $[x_1, x_2]= x_4, [x_2, x_1]=-x_4, [x_2, x_2]=x_4, [x_3, x_3]=x_4$.
\item[5] $[x_1, x_2]= x_4, [x_2, x_1]=cx_4, [x_3, x_3]=x_4, \quad c\in \cc\backslash\{1, -1\}$.
\item[6] $[x_1, x_1]=x_4, [x_2, x_2]=x_4, [x_3, x_3]=x_4$.
\end{description}
\end{thm}

\begin{proof} 
Let $A^2 = Leib(A) = {\rm span}\{x_4\}$. Then $[x_4, A] = 0$. Since $A$ is nilpotent and $\dim(A^2)=1$, we have $[A, x_4] = 0$. Let $V$ be the complementary subspace of $A^2$ in $A$. Then by Theorem \ref{cong}, the matrix of the bilinear form  $f( \ , \ )$ on $V$ with respect to an ordered basis $\{x_1, x_2, x_3\}$ is one of the following matrices where $c\neq1, -1$.

\begin{equation*}
 (i) \left( \begin{array}{ccc}
0 & 0 & 1\\
0 & 0 & 0 \\
0 & 1 & 0 \end{array} \right)\qquad
(ii) \left( \begin{array}{ccc}
0 & 0 & 1\\
0 & 1 & 1 \\
1 & -1 & 0\end{array} \right) \qquad
(iii) \left( \begin{array}{ccc}
0 & 1 & 0\\
-1 & 0 & 0 \\
0 & 0 & 1 \end{array}  \right) \]

\[(iv) \left( \begin{array}{ccc}
0 & 1 & 0\\
-1 & 1 & 0 \\
0 & 0 &  1 \end{array} \right)\qquad
(v) \left( \begin{array}{ccc}
0 & 1 & 0\\
c & 0 & 0 \\
0 & 0 & 1 \end{array}\right) \qquad
(vi) \left( \begin{array}{ccc}
1 & 0 & 0\\
0 & 1 & 0 \\
0 & 0 & 1 \end{array} \right) 
\end{equation*}

Using the definition of the bilinear form $f( \ , \ )$ on $V$ we easily see that the isomorphism class corresponding to the 
matrix $(i)$ has the only nontrivial multiplications among basis vectors given in $(1)$. Similarly, the isomorphism classes
corresponding to the matrices $(ii), (iii), (iv), (v), (vi)$ are $(2), (3), (4), (5), (6)$ respectively.
\end{proof}

\begin{thm}
Let $A$ be a non-split non-Lie nilpotent Leibniz algebra of dimension $5$ with $\dim(A^2)=1$. Then $A$ is isomorphic to a Leibniz algebra spanned by 
$\{x_1, x_2, x_3, x_4, x_5\}$ with the nonzero products given by one of the following:

\begin{description}
\item[1] $[x_1, x_4]= x_5, [x_2, x_3]=x_5, [x_2, x_4]=cx_5, [x_3, x_2]=cx_5, [x_4, x_1]=cx_5, [x_4, x_2]=x_5 , \quad c\in \cc\backslash\{1, -1\}$.
\item[2] $[x_1, x_4]= x_5, [x_2, x_3]=x_5, [x_2, x_4]=x_5, [x_3, x_2]=x_5, [x_4, x_1]=x_5, [x_4, x_2]=-x_5$.
\item[3] $[x_1, x_4]= x_5, [x_2, x_3]=x_5, [x_2, x_4]=x_5, [x_3, x_2]=-x_5, [x_3, x_3]=x_5, [x_4, x_1]=-x_5, [x_4, x_2]=x_5$.
\item[4] $[x_1, x_3]= x_5, [x_3, x_2]=x_5, [x_4, x_4]=x_5$.
\item[5] $[x_1, x_3]= x_5, [x_2, x_2]=x_5, [x_2, x_3]=x_5, [x_3, x_1]=x_5, [x_3, x_2]=-x_5, [x_4, x_4]=x_5$.
\item[6] $[x_1, x_2]=x_5, [x_2, x_1]=-x_5, [x_3, x_4]=x_5, [x_4, x_3]=-x_5, [x_4, x_4]=x_5$.
\item[7] $[x_1, x_2]=x_5, [x_2, x_1]=-x_5, [x_3, x_4]=x_5, [x_4, x_3]=cx_5, \quad c\in \cc\backslash\{1, -1\}$.
\item[8] $[x_1, x_2]=x_5, [x_2, x_1]=-x_5, [x_2, x_2]=x_5, [x_3, x_4]=x_5, [x_4, x_3]=-x_5, [x_4, x_4]=x_5$
\item[9] $[x_1, x_2]=x_5, [x_2, x_1]=-x_5, [x_2, x_2]=x_5, [x_3, x_4]=x_5, [x_4, x_3]=cx_5, \quad c\in \cc\backslash\{1, -1\}$.
\item[10] $[x_1, x_2]=x_5, [x_2, x_1]=c_1x_5, [x_3, x_4]=x_5, [x_4, x_3]=c_2x_5, \quad c_1, c_2\in \cc\backslash\{1, -1\}$.
\item[11] $[x_1, x_2]=x_5, [x_2, x_1]=-x_5, [x_3, x_3]=x_5, [x_4, x_4]=x_5$.
\item[12] $[x_1, x_2]= x_5, [x_2, x_1]=-x_5, [x_2, x_2]=x_5, [x_3, x_3]=x_5, [x_4, x_4]=x_5$.
\item[13] $[x_1, x_2]= x_5, [x_2, x_1]=cx_5, [x_3, x_3]=x_5, [x_4, x_4]=x_5 \quad c\in \cc\backslash\{1, -1\}$.
\item[14] $[x_1, x_1]=x_5, [x_2, x_2]=x_5, [x_3, x_3]=x_5, [x_4, x_4]=x_5$.
\end{description}
\end{thm}

\begin{proof}
Let $A^2 = Leib(A) = {\rm span}\{x_5\}$. Then $[x_5, A] = 0$. Since $A$ is nilpotent and $\dim(A^2)=1$, we have $[A, x_5] = 0$. Let $V$ be the complementary subspace of $A^2$ in $A$. Then by Theorem \ref{cong}, there exists an ordered basis $\{x_1, x_2, x_3, x_4\}$ of 
$V$ such that the matrix of the bilinear form  $f( \ , \ )$ on $V$ is one of the following listed below. Here we group the matrices corresponding to each partition of $4$. 

\begin{equation*}
\begin{scriptsize}
\begin{array}{l | cr}
\mbox{   4 }  &    \quad  \quad   B_4, D_4, E_4 \\ \hline
\mbox{ 3+1 }  &    \quad  \quad   A_3\oplus 1, C_3\oplus 1 \\ \hline
\mbox{   2+2 }  &    \quad  \quad   F_2\oplus F_2, F_2\oplus E_2, F_2\oplus B_2, E_2\oplus E_2, E_2\oplus B_2, B_2\oplus B_2 \\ \hline
\mbox{ 2+1+1 }  &    \quad  \quad   F_2\oplus 1\oplus 1, E_2\oplus 1\oplus 1, B_2\oplus 1\oplus 1\\ \hline
\mbox{ 1+1+1+1} & \quad \quad 1\oplus 1 \oplus 1 \oplus 1  \\ \hline
 \end{array} 
\end{scriptsize}
\end{equation*}
Now $\{x_1, x_2, x_3, x_4, x_5\}$ is an ordered basis for $A$ and we have an isomorphism class corresponding to each matrix of the bilinear form 
$f( \ , \ )$ on $V$ listed above. Thus we have $14$ isomorphism classes with the nonzero multiplications among basis vectors given in the 
statement of the theorem since the algebra corresponding to $F_2 \oplus F_2$ is a Lie algebra.
\end{proof}

\begin{thm}
Let $A$ be a non-split non-Lie nilpotent Leibniz algebra of dimension $6$ with $\dim(A^2)=1$. Then $A$ is isomorphic to a Leibniz algebra spanned by $\{x_1, x_2, x_3, x_4, x_5, x_6\}$ with the nonzero products given by one of the following:
\begin{description}
\item[1] $[x_1, x_4]= x_6, [x_2, x_5]=x_6, [x_4, x_2]=x_6, [x_5, x_3]=x_6$.
\item[2] $[x_1, x_5]= x_6, [x_2, x_4]=x_6, [x_2, x_5]=x_6, [x_3, x_3]=x_6, [x_3, x_4]=x_6, [x_4, x_2]=x_6, [x_4, x_3]=-x_6, \newline [x_5, x_1]=x_6, [x_5, x_2]=-x_6 $.
\item[3] $[x_1, x_4]= x_6, [x_2, x_3]=x_6, [x_2, x_4]=cx_6, [x_3, x_2]=cx_6, [x_4, x_1]=cx_6, [x_4, x_2]=x_6, [x_5, x_5]=x_6, \newline c \in \cc\backslash\{1, -1\}$.
\item[4] $[x_1, x_4]= x_6, [x_2, x_3]=x_6, [x_2, x_4]=x_6, [x_3, x_2]=x_6, [x_4, x_1]=x_6, [x_4, x_2]=-x_6, [x_5, x_5]=x_6 $.
\item[5] $[x_1, x_4]= x_6, [x_2, x_3]=x_6, [x_2, x_4]=x_6, [x_3, x_2]=-x_6,[x_3, x_3]=x_6, [x_4, x_1]=-x_6, [x_4, x_2]=x_6, \newline [x_5, x_5]=x_6 $.
\item[6] $[x_1, x_3]= x_6, [x_3, x_2]=x_6, [x_4, x_5]=x_6, [x_5, x_4]=-x_6$.
\item[7]$[x_1, x_3]= x_6, [x_3, x_2]=x_6, [x_4, x_5]=x_6, [x_5, x_4]=-x_6, [x_5, x_5]=x_6$.
\item[8] $[x_1, x_3]= x_6, [x_3, x_2]=x_6, [x_4, x_5]=x_6, [x_5, x_4]=cx_6, \quad c\in \cc\backslash\{1, -1\}$.
\item[9] $[x_1, x_3]= x_6, [x_2, x_2]=x_6, [x_2, x_3]=x_6, [x_3, x_1]=x_6, [x_3, x_2]=-x_6, [x_4, x_5]=x_6, [x_5, x_4]=-x_6$.
\item[10] $[x_1, x_3]= x_6, [x_2, x_2]=x_6, [x_2, x_3]=x_6, [x_3, x_1]=x_6, [x_3, x_2]=-x_6, [x_4, x_5]=x_6, [x_5, x_4]=-x_6, \newline [x_5, x_5]=x_6$.
\item[11] $[x_1, x_3]= x_6, [x_2, x_2]=x_6, [x_2, x_3]=x_6, [x_3, x_1]=x_6, [x_3, x_2]=-x_6, [x_4, x_5]=x_6, [x_5, x_4]=cx_6, \newline c\in \cc\backslash\{1, -1\}$.
\item[12] $[x_1, x_3]= x_6, [x_3, x_2]=x_6, [x_4, x_4]=x_6, [x_5, x_5]=x_6$.
\item[13] $[x_1, x_3]= x_6, [x_2, x_2]=x_6, [x_2, x_3]=x_6, [x_3, x_1]=x_6, [x_3, x_2]=-x_6, [x_4, x_4]=x_6, [x_5, x_5]=x_6$.
\item[14] $[x_1, x_2]=x_6, [x_2, x_1]=-x_6, [x_3, x_4]=x_6, [x_4, x_3]=-x_6, [x_5, x_5]=x_6$.
\item[15] $[x_1, x_2]=x_6, [x_2, x_1]=-x_6, [x_3, x_4]=x_6, [x_4, x_3]=-x_6, [x_4, x_4]=x_6, [x_5, x_5]=x_6$.
\item[16] $[x_1, x_2]=x_6, [x_2, x_1]=-x_6, [x_3, x_4]=x_6, [x_4, x_3]=cx_6, [x_5, x_5]=x_6, \quad c\in \cc\backslash\{1, -1\}$.
\item[17] $[x_1, x_2]=x_6, [x_2, x_1]=-x_6, [x_2, x_2]=x_6, [x_3, x_4]=x_6, [x_4, x_3]=-x_6, [x_4, x_4]=x_6, [x_5, x_5]=x_6$.
\item[18] $[x_1, x_2]=x_6, [x_2, x_1]=-x_6, [x_2, x_2]=x_6, [x_3, x_4]=x_6, [x_4, x_3]=cx_6, [x_5, x_5]=x_6, \quad c\in \cc\backslash\{1, -1\}$.
\item[19] $[x_1, x_2]=x_6, [x_2, x_1]=c_1x_6, [x_3, x_4]=x_6, [x_4, x_3]=c_2x_6, [x_5, x_5]=x_6, \quad c_1, c_2\in \cc\backslash\{1, -1\}$.
\item[20] $[x_1, x_2]=x_6, [x_2, x_1]=-x_6, [x_3, x_3]=x_6, [x_4, x_4]=x_6, [x_5, x_5]=x_6$.
\item[21] $[x_1, x_2]= x_6, [x_2, x_1]=-x_6, [x_2, x_2]=x_6, [x_3, x_3]=x_6, [x_4, x_4]=x_6, [x_5, x_5]=x_6$.
\item[22] $[x_1, x_2]= x_6, [x_2, x_1]=cx_6, [x_3, x_3]=x_6, [x_4, x_4]=x_6, [x_5, x_5]=x_6, \quad c\in \cc\backslash\{1, -1\}$.
\item[23] $[x_1, x_1]=x_6, [x_2, x_2]=x_6, [x_3, x_3]=x_6, [x_4, x_4]=x_6, [x_5, x_5]=x_6$.
\end{description}
\end{thm}

\begin{proof}
Let $A^2 = Leib(A) = {\rm span}\{x_6\}$. Then $[x_6, A] = 0$. Since $A$ is nilpotent and $\dim(A^2)=1$, we have $[A, x_6] = 0$. Let $V$ be the complementary subspace of $A^2$ in $A$. Then by Theorem \ref{cong}, there exists an ordered basis $\{x_1, x_2, x_3, x_4, x_5\}$ of 
$V$ such that the matrix of the bilinear form  $f( \ , \ )$ on $V$ is one of the following listed below. Here we group the matrices corresponding to each partition of $5$. 

\begin{equation*}
\begin{scriptsize}
\begin{array}{l | cr}
\mbox{   5 }  &    \quad  \quad   A_5, C_5 \\ \hline
\mbox{ 4+1 }  &    \quad  \quad   B_4\oplus 1, D_4\oplus 1, E_4\oplus 1 \\ \hline
\mbox{   3+2 }  &    \quad  \quad   A_3\oplus F_2, A_3\oplus E_2, A_3\oplus B_2, C_3\oplus F_2, C_3\oplus E_2, C_3\oplus B_2 \\ \hline
\mbox{ 3+1+1 }  &    \quad  \quad   A_3\oplus 1\oplus 1, C_3\oplus 1\oplus 1 \\ \hline
\mbox{   2+2+1 }  &    \quad  \quad   F_2\oplus F_2 \oplus 1, F_2\oplus E_2 \oplus 1, F_2\oplus B_2 \oplus 1, E_2\oplus E_2 \oplus 1, E_2\oplus B_2 \oplus 1, B_2\oplus B_2 \oplus 1  \\  \hline
\mbox{ 2+1+1+1 }  &    \quad  \quad  F_2\oplus 1\oplus 1 \oplus 1,  E_2\oplus 1\oplus 1 \oplus 1, B_2\oplus 1\oplus 1 \oplus 1 \\  \hline
\mbox{ 1+1+1+1+1} & \quad \quad  1\oplus 1 \oplus 1 \oplus 1\oplus 1 \\ \hline
 \end{array}
\end{scriptsize}
\end{equation*}
Now $\{x_1, x_2, x_3, x_4, x_5, x_6\}$ is an ordered basis for $A$ and we have an isomorphism class corresponding to each matrix of the bilinear form 
$f( \ , \ )$ on $V$ listed above. Thus we have $23$ isomorphism classes with the nonzero multiplications among basis vectors given in the 
statement of this theorem.
\end{proof}

\begin{thm}
Let $A$ be a non-split non-Lie nilpotent Leibniz algebra of dimension $7$ with $\dim(A^2)=1$. Then $A$ is isomorphic to a Leibniz algebra spanned by $\{x_1, x_2, x_3, x_4, x_5, x_6, x_7\}$ with the nonzero products given by one of the following:
\begin{description}
\item[1] $[x_1, x_6]= x_7, [x_2, x_5]=x_7, [x_2, x_6]=cx_7, [x_3, x_4]=x_7, [x_3, x_5]=cx_7, [x_4, x_3]=cx_7, [x_5, x_2]=cx_7,
\newline    [x_5, x_3]=x_7, [x_6, x_1]=cx_7, [x_6, x_2]=x_7,  \quad c\in \cc\backslash\{1, -1\}$.
\item[2] $[x_1, x_6]= x_7, [x_2, x_5]=x_7, [x_2, x_6]=x_7, [x_3, x_4]=x_7, [x_3, x_5]=x_7, [x_4, x_3]=-x_7, [x_4, x_4]=x_7, \newline [x_5, x_2]=-x_7, [x_5, x_3]=x_7, [x_6, x_1]=-x_7, [x_6, x_2]=x_7$.
\item[3] $[x_1, x_6]= x_7, [x_2, x_5]=x_7, [x_2, x_6]=x_7, [x_3, x_4]=x_7, [x_3, x_5]=x_7, [x_4, x_3]=-x_7, [x_5, x_2]=-x_7, \newline [x_5, x_3]=x_7, [x_6, x_1]=-x_7, [x_6, x_2]=x_7$.
\item[4] $[x_1, x_4]= x_7, [x_2, x_5]=x_7, [x_4, x_2]=x_7, [x_5, x_3]=x_7, [x_6, x_6]=x_7$.
\item[5] $[x_1, x_5]= x_7, [x_2, x_4]=x_7, [x_2, x_5]=x_7, [x_3, x_3]=x_7, [x_3, x_4]=x_7, [x_4, x_2]=x_7, [x_4, x_3]=-x_7, \newline [x_5, x_1]=x_7, [x_5, x_2]=-x_7, [x_6, x_6]=x_7$.
\item[6] $[x_1, x_4]= x_7, [x_2, x_3]=x_7, [x_2, x_4]=cx_7, [x_3, x_2]=cx_7, [x_4, x_1]=cx_7, [x_4, x_2]=x_7, [x_5, x_6]=x_7, \newline [x_6, x_5]=-x_7, \quad c\in \cc\backslash\{1, -1\}$.
\item[7] $[x_1, x_4]= x_7, [x_2, x_3]=x_7, [x_2, x_4]=cx_7, [x_3, x_2]=cx_7, [x_4, x_1]=cx_7, [x_4, x_2]=x_7, [x_5, x_6]=x_7, \newline [x_6, x_5]=-x_7, [x_6, x_6]=x_7, \quad c\in \cc\backslash\{1, -1\}$.
\item[8]$[x_1, x_4]= x_7, [x_2, x_3]=x_7, [x_2, x_4]=c_1x_7, [x_3, x_2]=c_1x_7, [x_4, x_1]=c_1x_7, [x_4, x_2]=x_7, \newline [x_5, x_6]=x_7, [x_6, x_5]=c_2x_7, \quad c_1, c_2\in \cc\backslash\{1, -1\}$.
\item[9] $[x_1, x_4]= x_7, [x_2, x_3]=x_7, [x_2, x_4]=x_7, [x_3, x_2]=x_7, [x_4, x_1]=x_7, [x_4, x_2]=-x_7, [x_5, x_6]=x_7, \newline [x_6, x_5]=-x_7$.
\item[10] $[x_1, x_4]= x_7, [x_2, x_3]=x_7, [x_2, x_4]=x_7, [x_3, x_2]=x_7, [x_4, x_1]=x_7, [x_4, x_2]=-x_7, [x_5, x_6]=x_7, \newline [x_6, x_5]=-x_7, [x_6, x_6]=x_7$.
\item[11] $[x_1, x_4]= x_7, [x_2, x_3]=x_7, [x_2, x_4]=x_7, [x_3, x_2]=x_7, [x_4, x_1]=x_7, [x_4, x_2]=-x_7, [x_5, x_6]=x_7, \newline [x_6, x_5]=cx_7, \quad c\in \cc\backslash\{1, -1\}$.
\item[12] $[x_1, x_4]= x_7, [x_2, x_3]=x_7, [x_2, x_4]=x_7, [x_3, x_2]=-x_7,[x_3, x_3]=x_7, [x_4, x_1]=-x_7, \newline [x_4, x_2]=x_7, [x_5, x_6]=x_7, [x_6, x_5]=-x_7$.
\item[13] $[x_1, x_4]= x_7, [x_2, x_3]=x_7, [x_2, x_4]=x_7, [x_3, x_2]=-x_7,[x_3, x_3]=x_7, [x_4, x_1]=-x_7, \newline [x_4, x_2]=x_7, [x_5, x_6]=x_7, [x_6, x_5]=-x_7, [x_6, x_6]=x_7$.
\item[14] $[x_1, x_4]= x_7, [x_2, x_3]=x_7, [x_2, x_4]=x_7, [x_3, x_2]=-x_7,[x_3, x_3]=x_7, [x_4, x_1]=-x_7, \newline [x_4, x_2]=x_7, [x_5, x_6]=x_7, [x_6, x_5]=cx_7, \quad c\in \cc\backslash\{1, -1\}$.
\item[15] $[x_1, x_4]= x_7, [x_2, x_3]=x_7, [x_2, x_4]=cx_7, [x_3, x_2]=cx_7, [x_4, x_1]=cx_7, [x_4, x_2]=x_7, \newline [x_5, x_5]=x_7, [x_6, x_6]=x_7 \quad c\in \cc\backslash\{1, -1\}$.
\item[16] $[x_1, x_4]= x_7, [x_2, x_3]=x_7, [x_2, x_4]=x_7, [x_3, x_2]=x_7, [x_4, x_1]=x_7, [x_4, x_2]=-x_7, [x_5, x_5]=x_7, \newline [x_6, x_6]=x_7$.
\item[17] $[x_1, x_4]= x_7, [x_2, x_3]=x_7, [x_2, x_4]=x_7, [x_3, x_2]=-x_7,[x_3, x_3]=x_7, [x_4, x_1]=-x_7, \newline [x_4, x_2]=x_7, [x_5, x_5]=x_7, [x_6, x_6]=x_7$.
\item[18]$[x_1, x_3]= x_7, [x_3, x_2]=x_7, [x_4, x_6]= x_7, [x_6, x_5]=x_7$.
\item[19] $[x_1, x_3]= x_7, [x_3, x_2]=x_7, [x_4, x_6]= x_7, [x_5, x_5]=x_7, [x_5, x_6]=x_7, [x_6, x_4]=x_7, [x_6, x_5]=-x_7$.
\item[20] $[x_1, x_3]= x_7, [x_2, x_2]=x_7, [x_2, x_3]=x_7, [x_3, x_1]=x_7, [x_3, x_2]=-x_7, [x_4, x_6]= x_7, [x_5, x_5]=x_7, \newline [x_5, x_6]=x_7, [x_6, x_4]=x_7, [x_6, x_5]=-x_7$.
\item[21] $[x_1, x_3]= x_7, [x_3, x_2]=x_7, [x_4, x_5]=x_7, [x_5, x_4]=-x_7, [x_6, x_6]=x_7$.
\item[22] $[x_1, x_3]= x_7, [x_3, x_2]=x_7, [x_4, x_5]=x_7, [x_5, x_4]=-x_7, [x_5, x_5]=x_7, [x_6, x_6]=x_7$.
\item[23] $[x_1, x_3]= x_7, [x_3, x_2]=x_7, [x_4, x_5]=x_7, [x_5, x_4]=cx_7, [x_6, x_6]=x_7, \quad c\in \cc\backslash\{1, -1\}$.
\item[24] $[x_1, x_3]= x_7, [x_2, x_2]=x_7, [x_2, x_3]=x_7, [x_3, x_1]=x_7, [x_3, x_2]=-x_7, [x_4, x_5]=x_7, [x_5, x_4]=-x_7, \newline [x_6, x_6]=x_7$.
\item[25] $[x_1, x_3]= x_7, [x_2, x_2]=x_7, [x_2, x_3]=x_7, [x_3, x_1]=x_7, [x_3, x_2]=-x_7, [x_4, x_5]=x_7, [x_5, x_4]=-x_7, \newline [x_5, x_5]=x_7, [x_6, x_6]=x_7$.
\item[26] $[x_1, x_3]= x_7, [x_2, x_2]=x_7, [x_2, x_3]=x_7, [x_3, x_1]=x_7, [x_3, x_2]=-x_7, [x_4, x_5]=x_7, [x_5, x_4]=cx_7, \newline [x_6, x_6]=x_7, \quad c\in \cc\backslash\{1, -1\}$.
\item[27] $[x_1, x_3]= x_7, [x_3, x_2]=x_7, [x_4, x_4]=x_7, [x_5, x_5]=x_7, [x_6, x_6]=x_7$.
\item[28] $[x_1, x_3]= x_7, [x_2, x_2]=x_7, [x_2, x_3]=x_7, [x_3, x_1]=x_7, [x_3, x_2]=-x_7, [x_4, x_4]=x_7, [x_5, x_5]=x_7, \newline [x_6, x_6]=x_7$.
\item[29] $[x_1, x_2]=x_7, [x_2, x_1]=-x_7, [x_3, x_4]=x_7, [x_4, x_3]=-x_7, [x_5, x_6]=x_7, [x_6, x_5]=-x_7, [x_6, x_6]=x_7$.
\item[30] $[x_1, x_2]=x_7, [x_2, x_1]=-x_7, [x_3, x_4]=x_7, [x_4, x_3]=-x_7, [x_5, x_6]=x_7, [x_6, x_5]=cx_7, \quad c\in \cc\backslash\{1, -1\}$.
\item[31] $[x_1, x_2]=x_7, [x_2, x_1]=-x_7, [x_3, x_4]=x_7, [x_4, x_3]=-x_7, [x_4, x_4]=x_7, [x_5, x_6]=x_7, \newline [x_6, x_5]=-x_7, [x_6, x_6]=x_7$.
\item[32] $[x_1, x_2]=x_7, [x_2, x_1]=-x_7, [x_3, x_4]=x_7, [x_4, x_3]=-x_7, [x_4, x_4]=x_7, [x_5, x_6]=x_7, \newline [x_6, x_5]=cx_7, \quad c\in \cc\backslash\{1, -1\}$.
\item[33] $[x_1, x_2]=x_7, [x_2, x_1]=-x_7, [x_3, x_4]=x_7, [x_4, x_3]=c_1x_7, [x_5, x_6]=x_7, [x_6, x_5]=c_2x_7, \newline c_1, c_2\in \cc\backslash\{1, -1\}$.
\item[34] $[x_1, x_2]=x_7, [x_2, x_1]=-x_7, [x_2, x_2]=x_7, [x_3, x_4]=x_7, [x_4, x_3]=-x_7, [x_4, x_4]=x_7, \newline [x_5, x_6]=x_7, [x_6, x_5]=-x_7, [x_6, x_6]=x_7$.
\item[35] $[x_1, x_2]=x_7, [x_2, x_1]=-x_7, [x_2, x_2]=x_7, [x_3, x_4]=x_7, [x_4, x_3]=-x_7, [x_4, x_4]=x_7, \newline [x_5, x_6]=x_7, [x_6, x_5]=cx_7, \quad c\in \cc\backslash\{1, -1\}$.
\item[36] $[x_1, x_2]=x_7, [x_2, x_1]=-x_7, [x_2, x_2]=x_7, [x_3, x_4]=x_7, [x_4, x_3]=c_1x_7, [x_5, x_6]=x_7, \newline [x_6, x_5]=c_2x_7, \quad c_1, c_2\in \cc\backslash\{1, -1\}$.
\item[37] $[x_1, x_2]=x_7, [x_2, x_1]=c_1x_7, [x_3, x_4]=x_7, [x_4, x_3]=c_2x_7, [x_5, x_6]=x_7, [x_6, x_5]=c_3x_7, \newline c_1, c_2, c_3\in \cc\backslash\{1, -1\}$.
\item[38] $[x_1, x_2]=x_7, [x_2, x_1]=-x_7, [x_3, x_4]=x_7, [x_4, x_3]=-x_7, [x_5, x_5]=x_7, [x_6, x_6]=x_7$.
\item[39] $[x_1, x_2]=x_7, [x_2, x_1]=-x_7, [x_3, x_4]=x_7, [x_4, x_3]=-x_7, [x_4, x_4]=x_7, [x_5, x_5]=x_7, [x_6, x_6]=x_7$.
\item[40] $[x_1, x_2]=x_7, [x_2, x_1]=-x_7, [x_3, x_4]=x_7, [x_4, x_3]=cx_7, [x_5, x_5]=x_7, [x_6, x_6]=x_7, \quad c\in \cc\backslash\{1, -1\}$.
\item[41] $[x_1, x_2]=x_7, [x_2, x_1]=-x_7, [x_2, x_2]=x_7, [x_3, x_4]=x_7, [x_4, x_3]=-x_7, [x_4, x_4]=x_7, \newline [x_5, x_5]=x_7, [x_6, x_6]=x_7$.
\item[42] $[x_1, x_2]=x_7, [x_2, x_1]=-x_7, [x_2, x_2]=x_7, [x_3, x_4]=x_7, [x_4, x_3]=cx_7, [x_5, x_5]=x_7, [x_6, x_6]=x_7, \newline c\in \cc\backslash\{1, -1\}$.
\item[43] $[x_1, x_2]=x_7, [x_2, x_1]=c_1x_7, [x_3, x_4]=x_7, [x_4, x_3]=c_2x_7, [x_5, x_5]=x_7, [x_6, x_6]=x_7,  \newline c_1, c_2\in \cc\backslash\{1, -1\}$.
\item[44] $[x_1, x_2]=x_7, [x_2, x_1]=-x_7, [x_3, x_3]=x_7, [x_4, x_4]=x_7, [x_5, x_5]=x_7, [x_6, x_6]=x_7$.
\item[45] $[x_1, x_2]=x_7, [x_2, x_1]=-x_7, [x_2, x_2]=x_7, [x_3, x_3]=x_7, [x_4, x_4]=x_7, [x_5, x_5]=x_7, [x_6, x_6]=x_7$.
\item[46] $[x_1, x_2]=x_7, [x_2, x_1]=cx_7, [x_3, x_3]=x_7, [x_4, x_4]=x_7, [x_5, x_5]=x_7, [x_6, x_6]=x_7, \quad c\in \cc\backslash\{1, -1\}$.
\item[47] $[x_1, x_1]=x_7, [x_2, x_2]=x_7, [x_3, x_3]=x_7, [x_4, x_4]=x_7, [x_5, x_5]=x_7, [x_6, x_6]=x_7$.
\end{description}
\end{thm}

\begin{proof}
Let $A^2 = Leib(A) = {\rm span}\{x_7\}$. Then $[x_7, A] = 0$. Since $A$ is nilpotent and $\dim(A^2)=1$, we have $[A, x_7] = 0$. Let $V$ be the complementary subspace of $A^2$ in $A$. Then by Theorem \ref{cong}, there exists an ordered basis $\{x_1, x_2, x_3, x_4, x_5, x_6\}$ of 
$V$ such that the matrix of the bilinear form  $f( \ , \ )$ on $V$ is one of the following listed below. Here we group the matrices corresponding to each partition of $6$. 

\begin{equation*}
\begin{scriptsize}
\begin{array}{l | cr}
\mbox{   6 }  &    \quad  \quad   B_6, E_6, F_6 \\ \hline

\mbox{ 5+1 }  &    \quad  \quad  A_5\oplus 1, C_5\oplus 1 \\ \hline

\mbox{   4+2 }  &    \quad  \quad   B_4\oplus F_2, B_4\oplus E_2, B_4\oplus B_2, D_4\oplus F_2, D_4\oplus E_2, D_4\oplus B_2, E_4\oplus F_2, E_4\oplus E_2, E_4\oplus B_2 \\ \hline

\mbox{ 4+1+1 }  &    \quad  \quad   B_4\oplus 1\oplus 1, D_4\oplus 1\oplus 1, E_4\oplus 1\oplus 1 \\ \hline

\mbox{   3+3 }  &    \quad  \quad   A_3\oplus A_3 , A_3\oplus C_3, C_3\oplus C_3  \\  \hline

\mbox{ 3+2+1 }  &    \quad  \quad  A_3 \oplus F_2\oplus 1, A_3 \oplus E_2\oplus 1, A_3 \oplus B_2\oplus 1, C_3 \oplus F_2\oplus 1, C_3 \oplus E_2\oplus 1, C_3 \oplus B_2\oplus 1 \\  \hline

\mbox{ 3+1+1+1} & \quad \quad A_3\oplus 1\oplus 1 \oplus 1, C_3\oplus 1\oplus 1 \oplus 1 \\ \hline

\mbox{ 2+2+2 }  &    \quad  \quad   F_2\oplus F_2\oplus F_2,  F_2\oplus F_2\oplus E_2, F_2\oplus F_2\oplus B_2, F_2\oplus E_2\oplus E_2, F_2\oplus E_2\oplus B_2,  \\ 
\mbox{ }  &    \quad  \quad F_2\oplus B_2\oplus B_2, E_2\oplus E_2\oplus E_2, E_2\oplus E_2\oplus B_2, E_2\oplus B_2\oplus B_2, B_2\oplus B_2\oplus B_2  \\  \hline

\mbox{ 2+2+1+1 }  &    \quad  \quad   F_2\oplus F_2\oplus 1\oplus 1, F_2\oplus E_2\oplus 1\oplus 1, F_2\oplus B_2\oplus 1\oplus 1,   \\
\mbox{ }  &    \quad  \quad  E_2\oplus E_2\oplus 1\oplus 1, E_2\oplus B_2\oplus 1\oplus 1, B_2\oplus B_2\oplus 1\oplus 1 \\  \hline

\mbox{ 2+1+1+1+1 }  &    \quad  \quad  F_2\oplus 1\oplus 1 \oplus 1\oplus 1, E_2\oplus 1\oplus 1 \oplus 1\oplus 1, B_2\oplus 1\oplus 1 \oplus 1\oplus 1 \\ \hline

\mbox{ 1+1+1+1+1+1 }  &    \quad  \quad   1\oplus 1\oplus 1 \oplus 1 \oplus 1 \oplus 1 \\ \hline
 \end{array}
\end{scriptsize}
\end{equation*}
Now $\{x_1, x_2, x_3, x_4, x_5, x_6, x_7\}$ is an ordered basis for $A$ and we have an isomorphism class corresponding to each matrix of the bilinear form 
$f( \ , \ )$ on $V$ listed above. Thus we have $47$ isomorphism classes with the nonzero multiplications among basis vectors given in the 
statement of the theorem since the algebra corresponding to $F_2 \oplus F_2 \oplus F_2$ is a Lie algebra.
\end{proof}

Now suppose $A$ is a non-split non-Lie nilpotent Leibniz algebra of dimension $8$ with $\dim(A^2)=1$. Let $A^2 = Leib(A) = {\rm span}\{x_8\}$. Then $[x_8, A] = 0$. Since $A$ is nilpotent and $\dim(A^2) = 1$, we have $[A, x_8] = 0$. Let $V$ be the complementary subspace of $A^2$ in $A$. Then by Theorem \ref{cong}, there exists an ordered basis $\{x_1, x_2, x_3, x_4, x_5, x_6, x_7\}$ of 
$V$ such that the matrix of the bilinear form  $f( \ , \ )$ on $V$ is one of the following listed below. Here we group the matrices corresponding to each partition of $7$.

\begin{equation*}
\begin{scriptsize}
\begin{array}{l | cr}

\mbox{   7 }  &    \quad  \quad   A_7, C_7 \\ \hline

\mbox{ 6+1 }  &    \quad  \quad   B_6\oplus 1, E_6\oplus 1, F_6\oplus 1 \\ \hline

\mbox{   5+2 }  &    \quad  \quad   A_5\oplus F_2, A_5\oplus E_2, A_5\oplus B_2, C_5\oplus F_2, C_5\oplus E_2, C_5\oplus B_2\\ \hline

\mbox{ 5+1+1 }  &    \quad  \quad  A_5\oplus 1\oplus 1, C_5\oplus 1\oplus 1 \\ \hline

\mbox{   4+3 }  &    \quad  \quad   B_4\oplus A_3 , B_4\oplus C_3, D_4\oplus A_3 , D_4\oplus C_3, E_4\oplus A_3 , E_4\oplus C_3, \\  \hline

\mbox{ 4+2+1 }  &    \quad  \quad   B_4 \oplus F_2\oplus 1, B_4 \oplus E_2\oplus 1, B_4 \oplus B_2\oplus 1, D_4 \oplus F_2\oplus 1, D_4 \oplus E_2\oplus 1,  \\ 
\mbox{ }  &    \quad  \quad  D_4 \oplus B_2\oplus 1, E_4 \oplus F_2\oplus 1, E_4 \oplus E_2\oplus 1, E_4 \oplus B_2\oplus 1, \\
 \hline

\mbox{ 4+1+1+1} & \quad \quad B_4\oplus 1\oplus 1 \oplus 1, D_4\oplus 1\oplus 1 \oplus 1, E_4\oplus 1\oplus 1 \oplus 1 \\ \hline

\mbox{ 3+3+1 }  &    \quad  \quad  A_3\oplus A_3\oplus 1, A_3\oplus C_3\oplus 1, C_3\oplus C_3\oplus 1 \\ \hline

\mbox{ 3+2+2 }  &    \quad  \quad   A_3\oplus F_2\oplus F_2, A_3\oplus F_2\oplus E_2, A_3\oplus F_2\oplus B_2, A_3\oplus E_2\oplus E_2, A_3\oplus E_2\oplus B_2,  A_3\oplus B_2\oplus B_2, \\ 
\mbox{ }  &    \quad  \quad   C_3\oplus F_2\oplus F_2, C_3\oplus F_2\oplus E_2, C_3\oplus F_2\oplus B_2, C_3\oplus E_2\oplus E_2, C_3\oplus E_2\oplus B_2, C_3\oplus B_2\oplus B_2 \\   \hline

\mbox{ 3+2+1+1 }  &    \quad  \quad   A_3 \oplus F_2\oplus 1\oplus 1, A_3 \oplus E_2\oplus 1\oplus 1, A_3 \oplus B_2\oplus 1\oplus 1, C_3 \oplus F_2\oplus 1\oplus 1,  \\  
\mbox{ }  &    \quad  \quad C_3 \oplus E_2\oplus 1\oplus 1, C_3 \oplus B_2\oplus 1\oplus 1 \\ \hline

\mbox{ 3+1+1+1+1 }  &    \quad  \quad  A_3 \oplus 1 \oplus 1 \oplus 1 \oplus 1,  C_3 \oplus 1 \oplus 1 \oplus 1 \oplus 1 \\ \hline

\mbox{ 2+2+2+1 }  &    \quad  \quad   F_2\oplus F_2\oplus F_2 \oplus 1, F_2\oplus F_2\oplus E_2 \oplus 1,   F_2\oplus F_2\oplus B_2 \oplus 1, F_2\oplus E_2\oplus E_2 \oplus 1,  F_2\oplus E_2\oplus B_2 \oplus 1, \\
\mbox{ }  &    \quad  \quad  F_2\oplus B_2\oplus B_2 \oplus 1, E_2\oplus E_2\oplus E_2 \oplus 1, E_2\oplus E_2\oplus B_2 \oplus 1, E_2\oplus B_2\oplus B_2 \oplus 1, B_2\oplus B_2\oplus B_2 \oplus 1, \\ \hline

\mbox{ 2+2+1+1+1 }  &    \quad  \quad   F_2\oplus F_2\oplus 1 \oplus 1\oplus 1, F_2\oplus E_2\oplus 1 \oplus 1\oplus 1, F_2\oplus B_2\oplus 1 \oplus 1\oplus 1, \\  
\mbox{ }  &    \quad  \quad E_2\oplus E_2\oplus 1 \oplus 1\oplus 1,  E_2\oplus B_2\oplus 1 \oplus 1\oplus 1, B_2\oplus B_2\oplus 1 \oplus 1\oplus 1, \\ \hline

\mbox{ 2+1+1+1+1+1 }  &    \quad  \quad   F_2\oplus 1\oplus 1 \oplus 1\oplus 1\oplus 1, E_2\oplus 1\oplus 1 \oplus 1\oplus 1\oplus 1, B_2\oplus 1\oplus 1 \oplus 1\oplus 1\oplus 1 \\ \hline

\mbox{ 1+1+1+1+1+1+1 }  &    \quad  \quad   1\oplus 1\oplus 1 \oplus 1 \oplus 1 \oplus 1 \oplus 1  \\ \hline
\end{array}
\end{scriptsize}
\end{equation*}
Now $\{x_1, x_2, x_3, x_4, x_5, x_6, x_7, x_8\}$ is an ordered basis for $A$ and we have an isomorphism class corresponding to each matrix of the bilinear form 
$f( \ , \ )$ on $V$ listed above. Then as before there are $74$ isomorphism classes (one corresponding to each matrix) of non-split non-Lie nilpotent Leibniz algebra of dimension $8$ with ${\rm dim}(A^2) = 1$.

\begin{thm} \label{solvable} Let $A$ be a non-split non-Lie non-nilpotent solvable Leibniz algebra with $\dim(A^2)=1$. Then $\dim(A)=2$ with basis $\{x_1, x_2\}$ and the nonzero multiplications given by:
\begin{center}
$[x_1, x_1]=x_2, [x_1, x_2]=x_2$.
\end{center}  
\end{thm}
\begin{proof}
Let $A$ be $n$-dimensional non-Lie non-nilpotent solvable Leibniz algebra with $\dim(A^2)=1$. Let $A^2=Leib(A)$ and $A^2={\rm span}\{x^{\prime}\}$ for some $0\neq x^{\prime}\in A$. Then for any $u, v\in A$, we have $[u, v]=cx^{\prime}$ for some $c\in \cc$. Define the bilinear form $f( \ , \ ): A\times A\rightarrow \cc$ by $f(u, v)=c$  for all $u, v\in A$. The canonical forms for the congruence classes of matrices associated with the bilinear form $f(\ ,\ )$ are as given in Theorem \ref{cong}. We choose an ordered basis $\{x_1, x_2, \ldots, x_n\}$ of $A$ such that $[x_i, x_j]=a_{ij}x^{\prime}$ where the congruence matrix $S=(a_{ij})^{n}_{i, j=1}$ is one of the matrix given in Theorem \ref{cong} or the direct sum of these matrices. Let $x^{\prime}=a_1x_1+ a_2x_2+ \cdots + a_nx_n\in A$ for some $a_1, a_2, \ldots, a_n\in\cc$. Observe that $[x^{\prime}, A]=0$. Since $A$ is a non-Lie Leibniz algebra the matrix $S\neq0$.
\par
If possible suppose  $S=K\oplus M$ where $K$ be a $k\times k$ and $M$ be $(n-k)\times (n-k)$ matrices given in Theorem \ref{cong}. Assume $K=0$ and $M\neq0$. Write $x^{\prime}=y+a_{k+1}x_{k+1}+\cdots+ a_nx_n$ where $y=a_1x_1+ a_2x_2+ \cdots + a_kx_k$. If $y=0$ choose $I_1={\rm span}\{x_1, x_2, \ldots, x_k\}$ and $I_2={\rm span}\{x_{k+1}, x_{k+2}, \ldots, x_n\}$. If $y\neq0$, let $W$ be a complementary subspace to ${\rm span}\{y\}$ in ${\rm span}\{x_1, x_2, \ldots, x_k\}$. Choose an ordered basis $\{z_1, z_2, \ldots, z_{k-1}\}$ for $W$. Notice that the matrix $S$ is not changed after this choice. Let $I_1={\rm span}\{z_1, z_2, \ldots, z_{k-1}\}$ and $I_2={\rm span}\{x_{k+1}, x_{k+2}, \ldots, x_n, y\}$. Then $I_1$ and $I_2$ are ideals and $A=I_1\oplus I_2$ which is a contradiction since $A$ is non-split. Similarly $K\neq0$ and $M=0$ leads to a contradiction.
\par
Now assume $K\neq0$ and $M\neq0$. Then $[x_r, x_s]=\alpha x^{\prime}, [x_t, x_u]=\beta x^{\prime}$ for some $1\le r, s\le k, \ k+1\le t, u\le n, \alpha\neq0, \beta\neq0$. Furthermore, since $A$ is non-nilpotent solvable Leibniz algebra we have $[z, x^{\prime}]=\gamma x^{\prime}$ for some $z\in A, \gamma\neq0$. By Leibniz identity $\alpha\gamma x^{\prime}= [z, [x_r, x_s]]=[[z, x_r], x_s]+ [x_r, [z, x_s]]= [x_r, [z, x_s]]$ since $[z, x_r]\in Leib(A)$. Therefore, $[z, x_s]\neq0$ and hence $[x_r, x^{\prime}]\neq0$. Again by Leibniz identity $\beta [x_r, x^{\prime}]= [x_r, [x_t, x_u]]=[[x_r, x_t], x_u]+ [x_t, [x_r, x_u]]=0$ since $[x_r, x_t]=0=[x_r, x_u]$. This is a contradiction since $\beta\neq0, [x_r, x^{\prime}]\neq0$.
\par
Hence the matrix $S$ of $f( \ , \ )$ is indecomposable and it is one of the matrices listed in Theorem \ref{cong}. Since $[x^{\prime}, A]=0$ the matrix $S$ must be singular. So by Theorem \ref{cong},  $S=A_{2k+1}(k\neq0) \ \rm{or} \ B_{2k}(c=0)$ since $A_1=0$ and other matrices are nonsingular.
\par
Suppose $S=A_{2k+1}(k\neq0)$. Note that here $n=2k+1$. In this case we have the following nonzero products:  $[x_1, x_{k+2}]=[x_{2k+1}, x_{k+1}]=[x_i, x_{k+i+1}]=[x_{k+i}, x_i]= x^{\prime} = a_1x_1+ a_2x_2+ \cdots + a_{2k+1}x_{2k+1}$ for $2\le i\le k$. Since $[x_i, x_{k+2}]=0,$ and $ [x_1, x_i]\in Leib(A)$ for $2\le i\le 2k+1$ we have \quad $0=[x_1, [x_i, x_{k+2}]]=[[x_1, x_i], x_{k+2}]+ [x_i, [x_1, x_{k+2}]] = [x_i, [x_1, x_{k+2}]]$. Hence $0=[x_i, [x_1, x_{k+2}]] = [x_i, x^{\prime}] = [x_i, a_1x_1+ a_2x_2+ \cdots + a_{2k+1}x_{2k+1}]$  for all $2\le i\le 2k+1$. From these equations we get $a_2=a_3=\cdots=a_{k+1}=a_{k+3}=\cdots=a_{2k+1}=0$. Furthermore,  $[x_1, [x_2, x_{k+3}]]=[[x_1, x_2], x_{k+3}]+ [x_2, [x_1, x_{k+3}]] = 0$ gives $a_{k+2}=0$. Therefore, $x^{\prime}=a_1x_1\in Leib(A)$ which is a contradiction since $[x_1, x_{k+2}]\neq0$.
\par
Now suppose $S=B_{2k}(c=0)$ ($k\neq1$). Note that here $n=2k$. In this case  the nonzero products are:  $[x_1, x_{2k}]=[x_i, x_{2k+1-i}]=[x_{k+i}, x_{k-i+2}]= x^{\prime} = a_1x_1+ a_2x_2+ \cdots + a_{2k}x_{2k}$ for $2\le i\le k$. Hence for all $2\le i\le 2k$ we have
 $0=[x_1, [x_i, x_{2k}]]=[[x_1, x_i], x_{2k}]+ [x_i, [x_1, x_{2k}]]= [x_i, [x_1, x_{2k}]]$ since $[x_i, x_{2k}]=0$ and $[x_1, x_i]\in Leib(A)$. Then we have the equations $0=[x_i, [x_1, x_{2k}]]=[x_i, a_1x_1+ a_2x_2+ \cdots + a_{2k}x_{2k}]$ (for all $2\le i\le 2k$)  implying $a_2=a_3=\cdots=a_{2k-1}=0$. Also since $[x_1, [x_2, x_{2k-1}]]=[[x_1, x_2], x_{2k-1}]+ [x_2, [x_1, x_{2k-1}]] = 0$  we get $a_{2k}=0$. Thus $x^{\prime}=a_1x_1\in Leib(A)$ which is a contradiction since $[x_1, x_{2k}]\neq0$.
\par
Finally, suppose $S=B_{2}(c=0)$. In this case we have only one nonzero product:  $[x_1, x_2]= x^{\prime} = a_1x_1+a_2x_2$. Hence $a_2\neq0$ since $x^{\prime} \in Leib(A)$. 
Then $0=[x^{\prime}, x_2] = [a_1x_1+a_2x_2, x_2]=a_1x^{\prime}$ implies that $a_1=0$. Now applying the following change of basis:  $x^{\prime}_1=\frac{1}{a_2}x_1+a_2x_2, x^{\prime}_2=a_2x_2$, we observe that $A$ is the $2$-dimensional cyclic Leibniz algebra as stated.
\end{proof}

\bigskip
\section{Classification of 3-Dimensional Leibniz Algebras}
\par

In this section, we revisit the classification of $3-$dimensional non-Lie non-nilpotent solvable Leibniz algebras. By Theorem \ref{solvable}, there is no non-split non-Lie non-nilpotent solvable Leibniz algebra $A$ of dimension $3$ with $\dim(A^2)=1$. Hence it suffices to classify the $3-$dimensional non-Lie non-nilpotent solvable Leibniz algebras with $\dim(A^2)=2$.

\begin{thm} \label{solv-clas}
Let $A$ be a non-split non-Lie non-nilpotent solvable Leibniz algebra of dimension $3$ with $\dim(A^2)=2$. Then $A$ is spanned by $\{x, y, z\}$ with the nonzero products given by one of the following:
\begin{description}
\item[1] $[x, x]=z, [x, z]=y, [x, y]=y$.
\item[2] $[x, z]=\alpha z, [x, y]=y,     \quad  \alpha\in \cc \backslash \{0\}$.
\item[3] $[x, z]=y, [x, y]=-\frac{1}{4}z+y$.
\item[4] $[x, x]=z, [x, y]=y=-[y, x]$.
\item[5] $[x, z]=\alpha z, [x, y]=y=-[y, x],  \quad \alpha\in \cc \backslash \{0\}$
\item[6] $[x, x]=z, [x, y]=y=-[y, x], [x, z]= 2z, [y, y]=z$.
\end{description}
\end{thm}

\begin{proof} Since $A$ is a non-split non-Lie non-nilpotent solvable Leibniz algebra of dimension $3$ with $\dim(A^2)=2$ we have $\dim(Z(A))\le1$ and  $1\le \dim(Leib(A))\le 2$.

Suppose $\dim(Leib(A))=2$ and $\dim(Z(A))=1$. Since $Leib(A)\subseteq Z^l(A)$ and $Z^l(A)\neq A$ we have $Z(A)\subseteq Z^l(A)=Leib(A)=A^2$. Choose $0 \neq a \in A \setminus Leib(A)$ such that $Z(A)={\rm span}\{a^2 = [a,a]\}$. Extend this to a basis $\{a^2, b\}$ of $Leib(A) = A^2$. Then 
$A={\rm span}\{a, a^2, b\}$ with the nonzero products: $[a, a]=a^2, [a, b]=\alpha_1a^2+\alpha_2b$  where $\alpha_2\neq0$. Defining $x= \frac{1}{\alpha_2}a+b, y= \frac{\alpha_1}{\alpha_2}a^2+b, z= (\frac{1+\alpha_1\alpha_2}{\alpha^2_2})a^2+b$, we get $A = {\rm span}\{x, y, z\}$ with the nonzero products given in (1).

If $\dim(Leib(A))=2$ and $\dim(Z(A))=0$, then as before we choose basis $\{a^2, b\}$ for $Leib(A) = A^2$ and $A = {\rm span}\{a, a^2, b\}$ with nonzero products: $[a, a]=a^2, [a, a^2]=\alpha_1a^2+\alpha_2b, [a, b]=\beta_1a^2+\beta_2b$. Since $\dim(Z(A)) = 0$, $[a, \beta_2a^2-\alpha_2b]=(\beta_2\alpha_1-\alpha_2\beta_1)a^2$ implies $\beta_2\alpha_1-\alpha_2\beta_1\neq 0$. If $\alpha_2=0$ then  $\beta_2\alpha_1-\alpha_2\beta_1\neq0$ gives $\alpha_1, \beta_2 \neq0$.
If $\alpha_2=0$ and $\beta_1 =0$, then setting $x=\frac{1}{\beta_2}a-\frac{1}{\alpha_1 \beta_2}a^2, y=b, z=a^2$ we have $A = {\rm span}\{x, y, z\}$ with the nonzero products given in (2) where $\alpha = \frac{\alpha_1}{\beta_2}$. If $\alpha_2=0$ and $\beta_1 \neq 0$ and $\alpha_1 = \beta_2$ then choosing $x=\frac{1}{2\alpha_1}a-\frac{1}{2\alpha^2_1}a^2, y=\frac{\alpha_1+\beta_1}{2\alpha_1}a^2+ \frac{1}{2}b, z=a^2+b$ gives $A = {\rm span}\{x, y, z\}$ with the nonzero products given in (3).
On the other hand, if $\alpha_2=0$ and $\beta_1 \neq 0$ and $\alpha_1 \neq \beta_2$ then choosing $x=\frac{1}{\beta_2}a-\frac{1}{\alpha_1\beta_2}a^2, y=\frac{\beta_1}{\beta_2-\alpha_1}a^2+ b, z=\frac{1}{\beta^2_2} a^2$ gives $A = {\rm span}\{x, y, z\}$ with the nonzero products given in (2)
with $\alpha = \frac{\alpha_1}{\beta_2} \in\cc \backslash \{0, 1\}$.
If $\alpha_2 \neq 0$ and $\alpha_1 + \beta_2 = 0$ then choosing $x=\frac{1}{(\alpha_2\beta_1-\beta_2\alpha_1)^{1/2}}a-\frac{\alpha_1}{(\alpha_2\beta_1-\beta_2\alpha_1)^{3/2}}a^2-\frac{\alpha_2}{(\alpha_2\beta_1-\beta_2\alpha_1)^{3/2}}b,  y=[(\alpha_2\beta_1-\beta_2\alpha_1)^{1/2}-\alpha_1]a^2-\alpha_2b, z=[(\alpha_2\beta_1-\beta_2\alpha_1)^{1/2}+\alpha_1]a^2+\alpha_2b$ gives 
$A = {\rm span}\{x, y, z\}$ with the nonzero products given in (2) with $\alpha=-1$. 
If $\alpha_2 \neq 0$, $\alpha_1 + \beta_2 \neq 0$ and $\alpha_2\beta_1 -\beta_2\alpha_1 = -\frac{1}{4}(\alpha_1 + \beta_2)^2$ then choosing $x = \frac{1}{\alpha_1+ \beta_2}(a-\frac{\beta_2}{\beta_2\alpha_1-\alpha_2\beta_1}a^2+\frac{\alpha_2}{\beta_2\alpha_1-\alpha_2\beta_1}b), y =\alpha_1a^2+ \alpha_2b, z =(\alpha_1+ \beta_2)a^2$ gives $A = {\rm span}\{x, y, z\}$ with the nonzero products given in (3). Now suppose $\alpha_2 \neq 0$, $\alpha_1 + \beta_2 \neq 0$ and $\alpha_2\beta_1 -\beta_2\alpha_1 \neq -\frac{1}{4}(\alpha_1 + \beta_2)^2$. Choose $\alpha$ such that $-\frac{\alpha}{(1+\alpha)^2}=\frac{\alpha_2\beta_1-\beta_2\alpha_1}{(\alpha_1+\beta_2)^2}$. Note that $\alpha \in \cc \backslash \{-1, 0, 1\}$. Set  $x= \frac{1 + \alpha}{(\alpha_1+\beta_2)}(a-\frac{\beta_2}{\beta_2\alpha_1-\alpha_2\beta_1}a^2+\frac{\alpha_2}{\beta_2\alpha_1-\alpha_2\beta_1}b), \\ y=[-\frac{\alpha(\alpha_1+\beta_2)}{1-\alpha} + \frac{(1 + \alpha)\alpha_1}{(1-\alpha)}]a^2 + \frac{(1 + \alpha)\alpha_2}{(1-\alpha)}b, 
z=[-\frac{(1+\alpha)\alpha_1}{(1-\alpha)}+\frac{\alpha_1+\beta_2}{1-\alpha}]a^2-\frac{(1+\alpha)\alpha_2}{(1-\alpha)}b$.
Then $A = {\rm span}\{x, y, z\}$ with the nonzero products given in (2).

Now suppose $\dim(Leib(A))=1$ and $\dim(Z(A))=1$. Then $Leib(A) = Z(A)$. Assume that $Z(A) \neq Z^l(A)$. Since $Z^l(A) \neq A$, we have $\dim(Z^l(A)) = 2$. Choose $0 \neq a \in A \setminus Leib(A)$
such that $Leib(A) = Z(A) = {\rm span}\{a^2 = [a, a]\}$. Extend it to a basis $\{a^2, b\}$ for $Z^l(A)$. Since $a \nin Z^l(A)$, we have $A = {\rm span}\{a, a^2, b\}$ with the nonzero products: $[a, a]=a^2, [a, a^2]=\alpha a^2, [a, b]=\beta_1a^2+\beta_2b$. Then $[a+b, a+b]\in Leib(A)$ implies that $\beta_2=0$ which is a contradiction since $\dim(A^2)=2$. Hence $Leib(A) = Z(A) = Z^l(A) = {\rm span}\{a^2\}$. Extending it to a basis $\{a^2, b\}$ of $A^2$ we observe that $A = {\rm span}\{a, a^2, b\}$. Since $A/Leib(A)$ is a Lie algebra, we have $[a+Leib(A), b+Leib(A)]=-[b+Leib(A), a+Leib(A)]$ which implies $[b, a] + [a, b] \in Leib(A)$. Therefore, the nonzero products in $A$ are:  $[a, a]=a^2, [b, b]=\beta a^2, [a, b]=\alpha_1a^2+\gamma b, [b, a]=\alpha_2a^2-\gamma b$ where $\gamma \neq 0$ since $\dim(A^2) = 2$. Then the Leibniz identities 
 $[b, [a, b]]= [[b, a], b]+ [a, [b, b]]$ and $[a, [b, a]]= [[a, b], a]+ [b, [a, a]]$
 give $\alpha_2=-\alpha_1$ and $\beta=0$. Now setting $x=\frac{1}{\gamma}a, y=\alpha_1a^2+\gamma b, z=\frac{1}{\gamma^2}a^2$, we obtain $A = {\rm span}\{x, y, z\}$ with the nonzero products given in (4). 
 
 Finally, suppose $\dim(Leib(A))=1$ and $\dim(Z(A))=0$. As before choose $0 \neq a \in A \setminus Leib(A)$ such that $Leib(A) = {\rm span}\{a^2 = [a, a]\}$. Let $A^2 = {\rm span}\{a^2 ,b\}$. Then 
 $A = {\rm span}\{a, a^2, b\}$ with the nonzero products: $[a, a]=a^2, [a, a^2]=\delta a^2, [b, b]=\beta a^2, [b, a^2]=\gamma a^2, [a, b]=\alpha_1 a^2+ \theta b, [b, a]=\alpha_2a^2-\theta b$ since as before $[a, b] +[b, a] \in Leib(A)$. Since $\dim(A^2) = 2$ we have $\theta \neq 0$. Then the identity 
$[a, [b, a^2]]= [[a, b], a^2]+ [b, [a, a^2]]$
implies $\theta \gamma=0$, and hence $\gamma=0$ since $\theta\neq0$. Since $\dim(Z(A)) =0$ we have $\delta \neq 0$. Now it follows from the identity $[a, [b, a]]= [[a, b], a]+ [b, [a, a]]$ that 
$\theta(\alpha_1+\alpha_2) = \alpha_2\delta$. Since $A$ is solvable and $\dim(A^2) = \dim(A^{(2)}) = 2$, we have 
$\dim(A^{(3)}) \leq 1$. Suppose $\dim(A^{(3)}) = 1$. Then $\beta \neq 0$ and the identity $[b, [a, b]]= [[b, a], b]+ [a, [b, b]]$ gives $2\theta\beta=\delta\beta$ which implies $\theta=\frac{\delta}{2}$. Hence $\alpha_1=\alpha_2$
since $\theta(\alpha_1+\alpha_2) = \alpha_2\delta$. Now choosing  $x=\frac{2}{\delta}a+ (\frac{\beta\delta^2}{8}- \frac{2}{\delta^2})a^2,  y=\frac{\delta}{2} b- \alpha_1a^2,  z= \frac{\beta\delta^2}{4}a^2$
we have $A = {\rm span}\{x, y, z\}$ with the nonzero products given in (6). Now suppose 
$\dim(A^{(3)}) = 0$. Then $\beta = 0$. 
Choosing $x=\frac{1}{\theta}a- \frac{1}{\theta\delta}a^2, y=\theta b - \alpha_2a^2, z= a^2$ we obtain $A = {\rm span}\{x, y, z\}$ with the nonzero products given in (5) where $\alpha=\frac{\delta}{\theta}$. 

\end{proof}

\begin{rmk}
If $\alpha_1, \alpha_2\in \cc \backslash \{0\}$ such that $\alpha_1\neq \alpha_2$, then the corresponding two Leibniz algebras of type (2) are isomorphic if and only if $\alpha_1\alpha_2=1$. Rest of the algebras in Theorem \ref{solv-clas} are pairwise nonisomorphic.
\end{rmk}

\end{document}